%% file: morsetype.tex
\documentclass[12pt,a4paper]{article}
\usepackage{amsthm}
\usepackage{epsfig}
\usepackage{alltt}
\usepackage[latin1]{inputenc}
\usepackage{makeidx}
\usepackage{newlfont}
\usepackage{amsmath}
\usepackage{amssymb}
\usepackage{amsfonts}
\usepackage{amscd}
\usepackage{enumerate}
\usepackage{mathrsfs}
\usepackage{ifthen}

\input{makros}
\begin{document}

\author{Hannes Junginger-Gestrich\thanks{Address: {\sf Mathematisches Institut,
      Albert-Ludwigs-Universit\"at Freiburg, Abteilung f\"ur Reine Mathematik,
      Eckerstra{\ss}e 1, 79104 Freiburg im Breisgau (Germany)}, {\tt
      hannes.junginger-gestrich@math.uni-freiburg.de}.  This work was partially
    supported by the DFG Graduiertenkolleg ``Nichtlineare
    Differentialgleichungen''}}

\title{A Morse type uniqueness theorem for non-parametric minimizing
  hypersurfaces}
\maketitle
\begin{abstract}
  A classical result about minimal geodesics on $\DR^2$ with $\DZ^2$ periodic
  metric that goes back to H.M.~Morse's paper \cite{Mo} asserts that a minimal
  geodesic that is asymptotic to a periodic minimal geodesic cannot intersect
  any periodic minimal geodesic of the same period. This paper treats a similar
  theorem for nonparametric minimizing hypersurfaces without selfintersections
  -- as were studied by J.~Moser, V.~Bangert, P.H.~Rabinowitz, E.~Stredulinsky
  and others.
\end{abstract}
\section{Introduction}\label{graphintro}
The first progress to generalize the results of Morse \cite{Mo} and
G.A.~Hedlund \cite{He} -- who studied the case of $\DR^2$ with $\DZ^2$-periodic
metric -- on minimal geodesics on surfaces to higher dimension was made by Moser
\cite{Mos}. He observed that the key features of minimal geodesics on $T^2$ are
that they separate space and that they do not have selfintersections when
projected to $T^2 = \DR^2/DZ^2$. We point out that this last property is not
contained in the classical text and was proven in \cite{Ba88}. 

Amongst other theorems some of the classical results were generalized by Moser
to graphs of functions $u:\DR^n \to \DR$, which are minimizers of a
$\DZ^{n+1}$-periodic variational problem and are {\em without
  selfintersections}.  Below the setting is described precisely. Moser obtained
an a priori estimate that asserts that any such graph stays within universally
bounded Hausdorff distance to a plane, and he proved first existence results,
namely that for any given unit vector $\alpha \in\DR^{n+1}$ there exists such a
graph that is within finite Hausdorff distance to a plane with unit normal
$\alpha$.  H.~Koch, R.~de la Llave and C.~Radin, cf.~\cite{KodlLRa}, obtain
results of this type for functions on lattices.  A.~Candel and de la Llave
provide versions for functions on sets with more general group actions in
\cite{CadlL}.  In the framework of Moser, Bangert proves a fundamental
uniqueness result in \cite{Ba87} and he carries out a detailed investigation of
the minimal solutions in this framework in \cite{Ba89}. These result are
considered as a codimension one version of {\em Aubry-Mather Theory}. Together
with E.~Valdinoci we observed in \cite{JuVa} that the results in \cite{Ba89} are
related to a famous conjecture of E.~de Giorgi.  P.H.~Rabinowitz and
E.~Stredulinsky also investigated the Moser framework in \cite{RaSt03},
\cite{RaSt04a} and \cite{RaSt04b}. They utilize a renormalized functional and
find more complicated extremals -- so called multibump solutions.

A central point in \cite{Ba89} is Theorem \ref{graphunique}, cf.~\cite[Theorem
(6.6)]{Ba89}, however the proof given there is incomplete. With minor variations
we adopt the notation of \cite{Ba89} and give a completion of the proof. Our
strategy is inspired by Morse's proof. In \cite{Ju} we proved a version of this
theorem for parametric minimizing hypersurfaces, cf.~also \cite{Jupar}. Although
it is possible to prove the parametric result carrying over the method used
here, it is simpler and more natural to use the theory of (weak) calibrations.
It is an open question whether there exists a suitable concept of calibration
calibrating a given totally ordered family of nonselfintersecting minimizing
graphs. It would be desirable to find a calibration that is $\DZ^n$-invariant.

\subsection{Moser's variational problem and basic results}
Given an integrand $F: \DR^n\times\DR\times\DR^n\to \DR$, periodic in the first
$n+1$ variables, we study functions $u:\DR^n \to \DR$ that minimize the integral
$\int F(x,u,u_x)\,dx$ w.r.t. compactly supported variations. We assume $F \in
C^{2,\epsilon}(\DR^{2n+1})$ and that $F$ satisfies appropriate growth
conditions, cf.~\cite[(3.1)]{Mos}, ensuring the ellipticity of the corresponding
Euler-Lagrange equation. Under these conditions minimizers inherit regularity
from $F$ and are of class $C^{2,\epsilon}(\DR^n)$.  For $u:\DR^{n+1} \to \DR$
and $\bar k = (k,k') \in \DZ^{n+1}$, define $T_{\bar k}u:\DR^n\to \DR$ as
$$T_{\bar k}u(x) = u(x-k)+ k'\,.$$
Since $F$ is $\DZ^{n+1}$-periodic, $T$ determines a $\DZ^{n+1}$-action on the
set of minimizers.

We look at minimizers $u$ {\em without self-intersections}, i.e. for all $\bar k
\in \DZ^{n+1}$ either $T_{\bar k}u < u$ or $T_{\bar k}u = u$ or $T_{\bar k}u>
u$.  Equivalently one can require that the hypersurface $\graph (u) \subset
\DR^{n+1}$ has no self-intersections when projected into $T^{n+1} =
\DR^{n+1}/\DZ^{n+1}$.

We call minimizers without self-intersections shortly {\em solutions} and denote
the set of all solutions by $\mcm$. On $\mcm$ we consider the
$C^1_\loc$-topology.  For every $u \in \mcm$ \cite[Theorem 2.1]{Mos} shows that
$\graph(u)$ lies within universally bounded distance from a hyperplane.  We
define the {\em rotation vector} of $u$ is as the unit normal $\bar a_1(u) \in
\DR^{n+1}$ to this hyperplane, which has positive inner product $\bar a_1 \cdot
\bar e_{n+1}$ with the $(n+1)$st standard coordinate vector.\footnote{We remark
  that our notion of rotation vector differs slightly from this notion in
  \cite{Ba89}.} Another fundamental result of Moser, cf.~\cite[Theorem
3.1]{Mos}, implies that every $u \in \mcm$ is Lipschitz with constant depending
only on $\bar a_1(u)$ (and $F$).

If $\bar k \cdot \bar a_1$ is $>0 \;(<0)$, then $T_{\bar k}u> u \;(<u)$.  If
$\bar k \cdot \bar a_1 = 0$, both cases are possible. There is a complete
description in \cite[(3.3)--(3.7)]{Ba89}, that we subsume in
  \begin{prop}\label{graphbasic}
    For every $u \in \mcm$ there exists an integer $t = t(u) \in
    \{1,\ldots,n+1\}$ and unit vectors $\bar a_1= \bar a_1(u),\ldots,\bar a_t =
    \bar a_t(u)$, such that for $1\le s \le t$ we have
    \begin{align}\label{adm}
      \begin{split} 
        \bar a_s \in \lspan \bar \Gamma_s\,, \quad &\mbox{where } \bar \Gamma =
        \bar \Gamma_1 = \DZ^{n+1} \mbox{ and }\\&\bar \Gamma_s = \bar
        \Gamma_s(u) \mathrel{\mathop:}= \DZ^{n+1} \cap \lspan\{\bar a_1,\ldots,
        \bar a_{s-1}\}^\bot\,,
      \end{split}
    \end{align} and the $\bar a_1,\ldots,\bar a_t$ are uniquely determined by
    the following properties:
    \begin{enumerate}[{\em (i)}]
    \item $T_{\bar k} u > u$ if and only if there exists $1 \le s \le t$ such
      that $\bar k \in \bar \Gamma_s$ and $\bar k \cdot \bar a_s > 0$.
    \item $T_{\bar k}u = u$ if and only if $\bar k \in \bar \Gamma_{t+1}$.
    \end{enumerate}
  \end{prop}
  Moser proved in \cite{Mos} that, if $|\bar a_1| =1$ and $\bar a_1 \cdot \bar
  e_{n+1} > 0$, there exist functions $u \in \mcm$ with $\bar a_1(u) = \bar
  a_1$. A system of unit vectors $(\bar a_1,\ldots, \bar a_t)$ is called {\em
    admissible} if $\bar a_1 \cdot \bar e_{n+1} >0$ and relation \eqref{adm} is
  satisfied. For an admissible system $(\bar a_1,\ldots,\bar a_t)$ we write
  $$\mcm(\bar a_1,\ldots,\bar a_t) = \big\{u \in \mcm \mid t(u) = t\, \mbox{and
  } \bar a_s(u) = \bar a_s\;\mbox{for } 1\le s \le t\big \}\,.$$

  The following observation describes the action of subgroups of $\bar \Gamma$
  on solutions.
  \begin{prop}\label{4.2}
    If $u \in \mcm(\bar a_1,\ldots, \bar a_t),\, t > 1$, then there exist
    functions $u^-$ and $u^+$ in $\mcm(\bar a_1, \ldots, a_{t-1})$ with the
    following properties:
    \begin{enumerate}[{\em (a)}]
    \item If $\bar k_i \in \bar \Gamma_t$ and $\lim_{i \to \infty} \bar k_i
      \cdot \bar a_t = \pm \infty$ then $\lim_{i \to \infty} T_{\bar k_i} u =
      u^\pm$\,.
    \item $u^-<u<u^+$ and $T_{\bar k} u^- \ge u^+$ if $k \in \bar \Gamma_s$ and
      $\bar k \cdot \bar a_s > 0$ for some $1\le s <t$.
    \end{enumerate}
  \end{prop}
  \begin{proof}
    \cite[Proposition (4.2)]{Ba89}.
  \end{proof}
  
  Besides the fact that Theorem \ref{graphunique} below is of independent
  interest as uniqueness theorem, it is a central point in the proof of the
  following uniqueness and existence results, cf.~\cite[Sections 6 and 7]{Ba89}:

  If $(\bar a_1, \ldots, \bar a_t)$ is admissible, then $\mcm(\bar a_1, \ldots,
  \bar a_t)$ and even the (disjoint) union $\mcm(\bar a_1)\cup \mcm(\bar a_1,
  \bar a_2) \cup\ldots \cup \mcm(\bar a_1, \ldots, \bar a_t)$ are totally
  ordered. If $u_1, u_2 \in \mcm(\bar a_1, \ldots, \bar a_{t-1})$ satisfy
  $u_1<u_2$ and are neighbouring, i.e.~there exists no $u \in \mcm(\bar a_1,
  \ldots, \bar a_{t-1})$ with $u_1 < u < u_2$, then there exists $v \in
  \mcm(\bar a_1, \ldots, \bar a_t)$ with $u_1<v<u_2$.

  \section{The Uniqueness Theorem}
  \begin{theorem}\label{graphunique}
    Suppose $u \in \mcm(\bar a_1,\ldots, \bar a_t)$ and $t>1$. Then there is no
    $v \in \mcm(\bar a_1, \ldots, \bar a_{t-1})$ with $u^-<v<u^+$.
  \end{theorem} 

  For economical reasons it makes sense to use the following abbreviations for
  functions $u \in W^{1,2}_\loc(\DR^n)$ and $\phi \in W^{1,2}_{0}(\DR^n)$ and
  measurable sets $A \subset \DR^n$ (cf.~\cite{Mos} and \cite{Ba89}):

  \begin{align*} 
    I(u,A) &\mathrel{\mathop:}= \int_A F(x,u,u_x)\,dx \quad \mbox{ if this integral exists in }
    \DR \cup
    \{\pm \infty\}\,,\\
    \Delta (u,\phi,A) &\mathrel{\mathop:}= \int_A \big(F(x,u+\phi,u_x + \phi_x) -
    F(x,u,u_x)\big)\,dx \,.
  \end{align*}
  In order to prove the Theorem we will imitate Morse's proof of \cite[Theorem
  13]{Mo}. This is not straightforward because of several reasons:

  The proof is based on comparison arguments for which we need to find ``short''
  connections between solutions which are close (in $C^1_\loc$).  In the
  parametric case ``slicing'' from Geometric Measure Theory provides such short
  connections. In the non-parametric case we need connecting graphs, for which
  we can control the slope, because our variational problem punishes steepness.
  We extend the idea of \cite[Lemma (6.8)]{Ba87} of constructing such
  connections.

  In higher dimensions, we have to cope with two additional difficulties:
  Solutions could show different behaviour in different directions in view of
  Proposition \ref{graphbasic}: A solution $u$ might be {\em recurrent} in some
  directions, {\em periodic} in some directions and {\em heteroclinic} in some
  directions (cf.~\cite{Ba87} and \cite{Ba89}). Furthermore we can, in general,
  say nothing about how the hypersurfaces under consideration do intersect.
  
  \subsection*{Proof of Theorem \ref{graphunique} for $n=1$}

  In case $n=1$ we carry over Morse's technique to the non-parametric case. The
  proof in this case also serves as a guideline for the proof in case $n \ge 2$.

  Suppose there exists a function $v \in \mcm(\bar a_1)$ with $u^- < v < u^+$.
  Following \cite[proof of Theorem (6.6)]{Ba89}, we choose the generator $\bar
  k_0 = (k_0,k_0')$ of $\bar \Gamma_2 = \bar \Gamma_2(u)$ with $\bar k_0 \cdot
  \bar a_2(u) > 0$ and define
  $$ w = \max\big(u, \min(v, T_{\bar k_0}u)\big)\,,$$
  cf.~figure \ref{morsefig} on page \pageref{morsefig}. Clearly $k_0 \neq 0$.
  Without loss of generality we {\bf assume} that $k_0 < 0$.
  \begin{remark}
    Why the proofs for $n=1$ and $n\ge2$ are different: The function $w$ (also
    in the higher dimensional case) is defined using $T_{\bar k_0}u$ and $k_0$
    determines a one dimensional subspace $\DR k_0\subset \DR^n$.  We have to
    compare the energies of the functions $u$ and $w$ on domains that feature
    some periodicity in this direction. In case $n=1$ we can use intervals, but
    in case $n\ge 2$ round balls are not suitable and, in view of Lemma
    \ref{l69}, cuboids are also not suitable. We use cylinders with caps (the
    sets $Z(r,t)$ below).  Also the fact that $\DR k_0 \subsetneq \DR^n$ for $n
    \ge 2$ makes a finer investigation necessary, cf.~\eqref{eq:ac}.
  \end{remark}

  The Maximum Principle, cf.~e.g.~\cite[Lemma 4.2]{Mos}, implies that $w$ is not
  minimizing. So we can save energy by a compactly supported variation.  This
  observation is contained in the following lemma, which is a special case of
  Lemma \ref{6.8} and proven in \cite[(6.8)]{Ba89}:
  \begin{lemma}\label{6.8.1} 
    There exist $\delta > 0$ and $r_0> 0$ and a function $\psi \in
    W^{1,2}_0(\DR)$ with $\spt \psi \subset (-r_0,r_0)$ such that
    $$\Delta\big(w,\psi,(-r_0,r_0)\big) < - \delta \,.$$  
  \end{lemma}
  
  What is missing in the proof of \cite[Theorem (6.6)]{Ba89} is the construction
  of a variation $u+\phi$ of $u$ (with $\spt \phi$ contained in a compact
  interval $K$), that coincides with $w$ on $(-r_0,r_0)$ such that
  $I(u+\phi,K)-I(u,K)$ is smaller than the gain $\delta$ provided by Lemma
  \ref{6.8.1}, say smaller than $\frac \delta2$:
  \begin{lemma}\label{glss.1}
    For $\delta>0$ and $r_0>0$ from Lemma \ref{6.8.1} there exist a compact set
    $K \supset (-r_0,r_0)$ and a function $\phi \in W^{1,2}_{0}(\DR^n)$ with
    $\spt \phi\subset K$ such that $(u+\phi)\big|_{(-r_0,r_0)} =
    w\big|_{(-r_0,r_0)}$ and
    \begin{equation} 
      \label{eq:gu1.1}  
      \Delta(u,\phi,K) < \frac \delta 2\,. 
    \end{equation} 
  \end{lemma} 
  The corresponding result for $n \ge 2$ is Lemma \ref{glss}.  Once this is
  established one easily gives the
  \begin{proof}[Proof of Theorem \ref{graphunique} for $n=1$, assuming Lemma
    \ref{glss.1}] 
    If there would exist such a function $v$, we could construct the function
    $w$, and the two lemmas above yield compactly supported functions $\psi$ and
    $\phi$ such that
    \begin{align*}
      \Delta\big(u,&\phi + \psi, K\big) = \Delta\big(u,\phi + \psi
      ,(-r_0,r_0)\big) + \Delta\big(u, \phi + \psi,
      K \setminus(-r_0,r_0)\big)\\
      &= \Delta\big(u + \phi, \psi, (-r_0,r_0)\big) +
      \Delta\big(u,\phi,(-r_0,r_0)\big) +
      \Delta\big(u,\phi,K\setminus(-r_0,r_0)\big) \\
      &=\Delta\big(w, \psi, (-r_0,r_0)\big) + \Delta\big(u,\phi,K\big) < -\delta
      + \frac \delta2 = -\frac\delta2 < 0\,,
    \end{align*}
    and this contradicts the minimality of $u$.
  \end{proof} 
  For the proof of Lemma \ref{glss.1} we shall need two results: The first of
  these, Lemma \ref{l69.1}, is a special case of \cite[Lemma (6.8) and Lemma
  (6.9)]{Ba87}, or Lemma \ref{l69} below.  If $\epsilon>0$ and $t>0$ are given,
  it allows us to construct the function $\phi$ such that
  $(u+\phi)\big|_{(-t,t)} = w\big|_{(-t,t)}$ and $\big| \Delta\big(u,\phi, \spt
  \phi \setminus (-t,t)\big)\big|< \epsilon$, i.e.~it is indeed what one would
  call a ``short connection''.  The second one is the non-parametric analogue of
  another result of Morse, cf.~\cite[Theorem 12]{Mo}, and asserts that the
  integral of a periodic solution over one period equals the energy of any other
  periodic solution with the same period over one period.
  \begin{lemma}\label{l69.1} 
    Consider $u_1,u_2:\DR \to \DR$ with Lipschitz constant $L$ and $t \in \DR^+$
    and suppose $0 \le u_2 - u_1 \le C$ for some $C >0$. Then there exists a
    function $g: \DR \to \DR$ such that
    \begin{enumerate}[{\em (a)}]
    \item $g$ is Lipschitz with constant $2L+1$,
    \item $g\big|_{[-t,t]} = u_2\big|_{[-t,t]}$,
    \item $g\big|_{\DR\setminus [-t-C,t+C]} = u_1\big|_{\DR \setminus
        [-t-C,t+C]}$,
    \item $ \mcl^1\bigg(\big\{x \in \DR \mid |x| \ge t,\, g(x) \neq u_1(x)\big\}\bigg)
      \le (u_2 - u_1)(-t)+(u_2 - u_1)(t)\,,$
    \item \label{e.1} there exists a constant $\tilde A = \tilde A(C,L,F)$
      such that
      \begin{align*}
        \bigg|\int_{\DR \setminus [-t,t]} \big(F(x,g,g_x) -
        F(&x,u_1,(u_1)_x)\big)\,dx
        \bigg| \\
        &\le \tilde A \bigg((u_2 - u_1)(-t)+(u_2 - u_1)(t)\bigg)\,.
      \end{align*}
    \end{enumerate}
  \end{lemma}
  \begin{remark}\label{e2.1}
   Analogous statements are true if $0 \le u_1-u_2 \le C$.
  \end{remark}
  \begin{proof}
    Let $\operatorname{pr}: \DR \to [-t,t]$ be the nearest point projection and
    define
    $$ g(x) \mathrel{\mathop:}= \max\left\{u_2\big(\operatorname{pr}(x)\big) -
      (L+1)d\big(x,[-t,t]\big),u_1(x)\right\}\,.$$ One readily verifies that $g$
    satisfies (a)--(d). Since $F\big(x,h(x),h_x(x)\big)$ is uniformly bounded
    for all $x \in \DR$ and all $h \in \Lip(2L+1)$, also (e) follows.
  \end{proof}

  \begin{lemma}\label{vu.1}
    Consider the action $T'$ of $\DZ k_0$ on $\DR$, given by $T'_{k}
    x = x + k$ for every $k \in \DZ k_0$.  If $u_1, u_2 \in \mcm(\bar a_1)$ and
    $u_1\le u_2$ and $H_1, H_2$ are fundamental domains of $T'$, then
    $I(u_1,H_1) = I(u_2, H_2).$
  \end{lemma} 

  \begin{proof}
    Let $\epsilon > 0$ be given.  By the assumed periodicity of $u_1$ and $u_2$
    and the $\DZ^2$-periodicity of $F$, we may assume without loss of generality
    that
    $$H_1 = H_2 = \{x \in \DR \mid 0 \le x <|k_0|\}=\mathrel{\mathop:}H_0\,.$$
    
    By periodicity of $u_1$ and $u_2$ there exists a constant $C >0$ such that
    $u_2- u_1 \le C$.  Let $n \in \DN$ be such that $\frac 1n \tilde A C <
    \epsilon$ and set $t= n |k_0|$. Let $g$ be the function provided by Lemma
    \ref{l69.1}.  For $\phi = g -u_1$ we have $u_1+\phi = u_2$ on $(-t,t)$ and,
    by minimality of $u_1$,
    $$ I\big(u_1, (-(t+C),t+C)\big) \le  I\big(u_1 + \phi,
    (-(t+C),t+C)\big)\,.$$ Using Remark \ref{e2.1} and Lemma \ref{l69.1}(e),
    we obtain
    $$ \big| I\big(u_1, (-t,t)\big) - I\big(u_2, (-t,t)\big)\big| \le  2\tilde A
    C\,.$$ Then, by the assumed periodicity of $u_1$ and $u_2$,
    \begin{align*} 
      2n|I(u_1,H_0) - I(u_2, H_0)| &= \big| I\big(u_1, -t,t)\big) - I\big(u_2,
      (-t,t)\big)\big| \le 2\tilde A C\,,
    \end{align*}
    and thus $|I(u_1,H_0) - I(u_2, H_0)|<\epsilon$.
  \end{proof}

  \begin{proof}[Proof of Lemma \ref{glss.1}]
    According to Proposition \ref{4.2}(a) it is true that $T_{n\bar k_0}u \to
    u^\pm$ in $C^1_\loc$ as $n \to \pm \infty$. Thence 
    $$\big(w - u\big)(-t) + \big(w-u\big)(t) \to 0 \quad \mbox{ as } t \to
    \infty\,.$$

    Let $g_t$ be the functions provided by Lemma \ref{l69.1} for $u_1 = u$, $u_2
    =w$ and $t>r_0$ for $r_0$ from Lemma \ref{6.8.1}.  Set $\phi_t
    \mathrel{\mathop:}= g_t - u$ and $K_t = \spt \phi_t$.  Then, by Lemma
    \ref{l69.1} (e), we may choose $t_0$ so large that for $t \ge t_0$
    \begin{equation}
      \label{eq:a2.1}
      \big|\Delta\big(u,\phi_t, K_t \setminus (-t,t)\big)\big|<\frac \delta4\,.
    \end{equation}
    This estimates the ``cost of energy by short connections'' outside $(-t,t)$.
   \begin{figure}
      \includegraphics[scale=1]{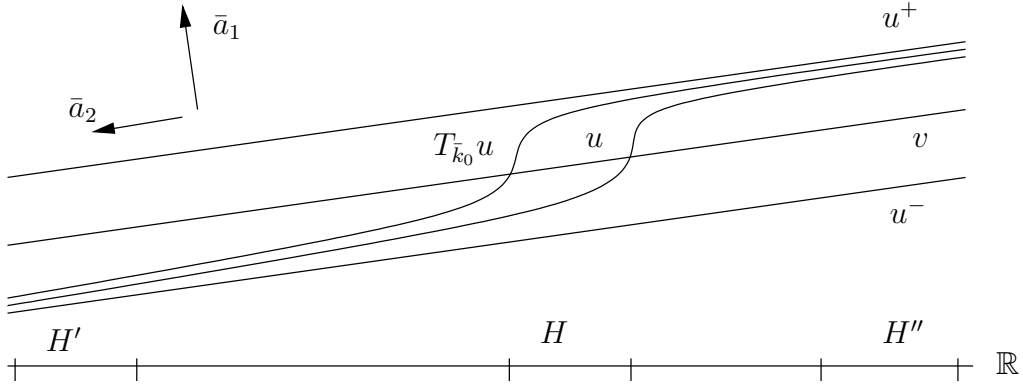}
      \caption{Idea for the proof of Theorem \ref{graphunique}}
      \label{morsefig}
    \end{figure}

    Now we have to compare the energy of $u$ and $w$ inside $(-t,t)$. We will
    have to consider the following fundamental domains of $T'$ (recall that we
    assume $k_0 <0$, and cf.~Figure \ref{morsefig}):
    \begin{align*} 
      H_t' &\mathrel{\mathop:}=(-t,-t-k_0]\\
      H_t'' &\mathrel{\mathop:}=[t +k_0,t)\\
      H &\mathrel{\mathop:}= \{x\in \DR\mid u(x)<v(x)\le T_{\bar k_0}u(x)\}\,.
    \end{align*}
    By continuity of $F$, the $C^1_\loc$-convergence provided by Proposition
    \ref{4.2}(a) implies $\big|I(u,H_t') - I(u^-,H_t')\big| \to 0 $ as $t \to
    \infty$. Hence by Lemma \ref{vu.1} we may choose $t'\ge t_0$ so large that
    for every $t \ge t'$ we have
    \begin{equation}
      \label{eq:a3.1}
      \big|I(v,H) - I(u,H_t')\big| < \frac \delta4\,.
    \end{equation}
    By periodicity of $v$ and $u^-$ and by the above-mentioned $C^1_\loc$-, and
    hence $C^0_\loc$-convergence, there exists $t'' \ge t'$ such that
    \begin{equation}\label{eq:conv.1} H_t'\cap \{u \ge v\} = \emptyset\,,\;
      H_t'' \cap \{T_{\bar k_0}u < v\} = \emptyset\;\mbox{ and } \; H \cap
      (-t,t) = H
    \end{equation}
    for all $t \ge t''$. Consequently, for $t\ge t''$, there is the
    decomposition
    \begin{align}
      \begin{split}\label{wdecomp.1}
        w\cdot\chi_{(-t,t)} &= v\cdot\chi_H + u\cdot\chi_{(-t,t)\cap \{u \ge v\}} + T_{\bar
          k_0}
        u\cdot\chi_{(-t,t) \cap \{T_{\bar k_0}u<v\}}\\
        &= v\cdot\chi_{H} + u\cdot\chi_{H_t'\cap \{u \ge v\}} +
        u\cdot\chi_{((-t,t)\setminus H_t')\cap \{u \ge v\}} \\
        &\quad + T_{\bar k_0} u\cdot\chi_{H_t''\cap \{T_{\bar k_0}u < v\}} + T_{\bar
          k_0} u\cdot\chi_{((-t,t) \setminus H_t'')\cap \{T_{\bar k_0}u<v\}}\,.
      \end{split}
    \end{align}
    Furthermore periodicity of $F$ and $v$ yields
    \begin{equation}\label{eq:a6.1}
      I\big(T_{\bar k_0} u,  ((-t,t) \setminus H_t'')\cap \{T_{\bar
        k_0}u<v\}\big) = I\big(u,((-t,t) \setminus H_t') \cap \{u<v\}
      \big)\,.
    \end{equation}
    From the decomposition \eqref{wdecomp.1} for $w$ we deduce for $t \ge t''$,
    using \eqref{eq:a3.1}, \eqref{eq:conv.1} and \eqref{eq:a6.1}:
    \begin{align*}
      I\big(w,(-t,t)\big) &< I\big(u, H_t'\big) + \frac \delta4+ 0 + I\big(u,
      ((-t,t)\setminus H_t') \cap \{u \ge v\}\big)\\
      &\quad +0+ I\big(u,((-t,t)\setminus H_t') \cap \{u<v\}\big)\\
      &=I\big(u,(-t,t)\big) + \frac \delta4\,.
    \end{align*}
    Together with \eqref{eq:a2.1} this gives $\Delta(u, \phi_t, K_t) < \frac
    \delta2\,.$
  \end{proof}

  \subsection*{Proof of Theorem \ref{graphunique} for $n\ge 2$}

  We assume the existence of $v \in \mcm(\bar a_1,\ldots, \bar a_{t-1})$ with
  $u^-<v<u^+$. As in the one-dimensional case we follow \cite{Ba89} and define
  the function $w$ as follows: Choose $\bar k_0=\big(k_0, (k_0)'\big) \in \bar
  \Gamma_t$ with $\bar k_0 \cdot \bar a_t > 0$, and set
  $$w = \max \big(u,\min (v,T_{\bar k_0} u)\big)\,.$$
  Let us write $j = \rank \bar \Gamma_t$. By \cite[(6.8)]{Ba89} we have the
  following
  \begin{lemma}\label{6.8}
    There exist $\delta > 0$ and $r_0>0$ such that for every $r>r_0$ there
    exists a function $\psi= \psi_r \in W^{1,2}_0(\DR^n)$ with $\spt \psi
    \subset B(0,r))$ such that

    $$\Delta\big(w,\psi,B(0,r)\big) < - \delta r^{j-1}\,.$$  
  \end{lemma}
  Here we will prove
  \begin{lemma}\label{glss}
    For every $r>0$ there exists $s\ge r$, a compact set $K = K_s \supset
    B(0,s)$ and a function $\phi = \phi_s \in W^{1,2}_{0}(\DR^n)$ with $\spt
    \phi \subset K$,\, $(u+\phi)\big|_{B(0,s)} = w\big|_{B(0,s)}$ such that for
    $\delta>0$ from Lemma \ref{6.8} we have
    \begin{equation}
      \label{eq:gu1}
      \Delta(u,\phi,K) < \frac \delta 2 s^{j-1}\,. 
    \end{equation}
  \end{lemma}
  \begin{proof}[Proof of Theorem \ref{graphunique}, assuming Lemma \ref{glss}]
    If there existed such a function $v$, we construct the function $w$, and the
    two Lemmas above yield compactly supported functions $\psi = \psi_s$ and
    $\phi = \phi_s,\, s > r_0,$ such that analogously to the case $n=1$
    \begin{align*}
      \Delta(u,\phi + \psi, K) < (-\delta + \frac \delta2) s^{j-1} =
      -\frac\delta2 s^{j-1} < 0\,,
    \end{align*}
    and this contradicts the minimality of $u$.
  \end{proof}

  We shall need a modification of the ``Slicing-Lemma'' \cite[Lemmas (6.8) and
  (6.9)]{Ba87}. This is necessary since we need this result not only for balls
  but also for sets featuring some periodicity in the direction of $k_0$, namely
  for the full ``cylinder with caps''
  \begin{equation*}
    Z(r,t) \mathrel{\mathop:}= \big\{ x \in \DR^n \mid d\big(x,\{\lambda k_0 \mid |\lambda| \le t\}\big)\le
    r\big\},\quad r > 0,\,t \in \DR^+ \cup \{\infty\} \,.
  \end{equation*}
  Let $C_t(r)$ denote the cylinder $\{x \in \DR^n\mid |x\cdot k_0| \le t \}\cap
  \partial Z(r,t)$ of radius $r$ and height $2t$ with ``soul'' $\DR k_0$.  Let
  $D_t(r)$ denote the set $\partial Z(r,t) \setminus C_t(r)$ that consists of
  two open $(n-1)$-half-spheres for $t<\infty$, and is empty if $t = \infty$.
  Note that $\partial Z(r,t) = C_t(r) \cup D_t(r)$ for every $r \in \DR^+, t \in
  \DR^+\cup \{\infty\}$.

  By $d\sigma$ we denote the $(n-1)$-dimensional area-element.

  \begin{lemma}\label{l69}
    Let $u_1,u_2:\DR^n \to \DR$ have Lipschitz constant $L$ and suppose $0 \le
    u_2 - u_1 \le C$ and $r \ge 1,\, t \in \DR^+ \cup \{\infty\}$. Then there
    exists a function $g: \DR^n \to \DR$ such that
    \begin{enumerate}[{\em (a)}]
    \item $g$ is Lipschitz with constant $2L+1$,
    \item $g = u_2$ inside $Z(r,t)$,
    \item $g = u_1$ outside $Z(r+C,t)$, which is compact if $t< \infty$,
    \item $\vol_n\left(\{x \in Z(r,t)^C \mid g(x) \neq u_1(x)\}\right)  \le(1+C)^{n-2}
      \int_{C_t(r)} (u_2 - u_1)(x) \,d\sigma(x)  \\ + (1+C)^{n-1} \int_{D_t(r)} (u_2 - u_1)(x)\,d\sigma(x)\,,$
    \item \label{e} there exists a constant $\tilde A = \tilde A(n,C,L,F)$
      such that
      \begin{align*}
        \bigg|&\int_{\DR^n\setminus Z(r,t)} \big(F(x,g,g_x) -
        F(x,u_1,(u_1)_x)\big)\,dx \bigg| \\
        &\le\tilde A \int_{C_t(r)} (u_2 - u_1)(x) \,d\sigma(x) + \tilde
        A \int_{D_t(r)} (u_2 - u_1)(x)\,d\sigma(x)\,.
      \end{align*}
    \end{enumerate}
  \end{lemma}
  
  \begin{remark}\label{e2}
    Analogous statements are true if $0 \le u_1-u_2\le C$.
  \end{remark}
  
  \begin{proof} We modify Bangert's proof. Let $\operatorname{pr}: \DR^n \to
    Z(r,t)$ be the nearest point projection and define
    $$ g(x) \mathrel{\mathop:}= \max\left\{u_2\big(\operatorname{pr}(x)\big) -
      (L+1)d\big(x,Z(r,t)\big),u_1(x)\right\}\,.$$ Hence $g$ satisfies (a) and
    (b). Since $u_1$ has Lipschitz constant
    $L$ we have
    $$u_1(x) \ge u_2(\pr(x)) +\big(u_1(\pr(x)) - u_2(\pr(x))\big) - Ld(x,Z(r,t))\,,$$
    and therefore $g(x) = u_1(x)$ if $d\big(x,Z(r,t)\big) \ge u_2(\pr(x)) -
    u_1(\pr(x))$ and $g$ satisfies (c).
  
    If $\nu_x$ denotes the outer unit normal to $\partial Z(r,t)$ we consider
    the transformation maps
    \begin{align*}
      &\tilde \tau: C_t(r)\times \DR^+ \to \DR^n,\, (x,s)\mapsto x + s\,\nu_x
      \quad
      \mbox{and}\\
      &\tau: D_t(r)\times \DR^+ \to \DR^n,\, (x,s) \mapsto x + s\,\nu_x\,,
    \end{align*}
    which occur in the following integration in cylindric and polar coordinates.
    Let $J\tilde \tau$ and $J \tau$ be the corresponding Jacobians.
    \begin{align*}
      \vol_n\big(\{x \in &Z(r,t)^C \mid g(x) \neq u_1(x) \}\big)  \\
      &\le \int_{C_t(r)} \int_r^{r+(u_2 - u_1)(x)}|J\tilde\tau(x,s)|\, ds\,
      d\sigma(x) \\
      &\qquad+ \int_{D_t(r)} \int_r^{r+(u_2 -
        u_1)(x)}|J\tau(x,s)|\,ds\,d\sigma(x) \\
      &\le (1+C)^{n-2} \int_{C_t(r)} (u_2 - u_1)(x) \,d\sigma(x) \\
      &\qquad+ (1+C)^{n-1} \int_{D_t(r)} (u_2 -
      u_1)(x)\,d\sigma(x) \\
    \end{align*}
    which is estimate (d). Since $F\big(x,h(x),h_x(x)\big)$ is uniformly bounded
    for all $x \in \DR^n$ and all $h \in \Lip(2L+1)$, we obtain (e).
  \end{proof}
we will
need the following simple observation:

\begin{lemma}\label{volslice} 
  Suppose $j \in \{0\} \cup \DN$ and $f:\DR^+ \to [0,\infty)$ is a measurable
  function, $r_0>0$ and $\int_0^r f(s)\,ds \le cr^j$ for a constant $c>0$ and
  every $r>r_0$.  Then, if $i \in \DN$ is such that $2^{i+1} \ge r_0$, we obtain
  for every $k \in \DN$
  \begin{equation*}
    \label{eq:volslice}
    \mcl^1\bigg( \big\{f(s) > 2^{j+1}ck\,s^{j-1}\big\} \cap [2^i,2^{i+1})\bigg) <
    \frac 1k\,2^i\,. 
  \end{equation*}
  Especially there exists a constant $\tilde c >0$ and a sequence $(s_i)_{i \in
    \DN}$ with $s_i \to \infty$ as $i \to \infty$ such that $f(s_i) < \tilde c\,
  s_i^{j-1}\,.$
\end{lemma}

\begin{proof}
  $(j = 0)$: If for $i \in \DN$ with $2^{i+1}\ge r_0$ the estimate was false,
  then
  $\int_{2^i}^{2^{i+1}}f(s)\,ds > \frac 1k  2^{i}\cdot 2 \cdot ck 2^{-(i+1)} = c\,,$
  which contradicts $\int_0^{2^{i+1}} f(s)\,ds\le c$.
  
  $(j \ge 1)$: If for $i \in \DN$ the estimate was not true, we calculate
    \begin{align*}
      c\,2^{(i+1)j}&\ge \int_0^{2^{i+1}} f(s)\,ds \ge \int_{2^i}^{2^{i+1}}f(s)\,ds\\
      &> \frac 1k 2^i\cdot 2^{j+1} ck \cdot 2^{i(j-1)} = c \cdot 2^{(i+1)j+1}
      \,.
    \end{align*}
    Division by $2^{(i+1)j}$ yields the contradiction $c > 2 c\,.$
  \end{proof}
 
  \begin{lemma}\label{vu}
    Consider the action $T'$ of $\DZ k_0$ on $\DR^n$, given by $T'_{k} x = x +
    k$ for every $k \in \DZ k_0$.  Consider $u_1, u_2 \in \mcm(\bar a_1, \ldots,
    \bar a_{t-1})$ with $u_1\le u_2$.  Suppose $T_{\bar k} u_1 \ge u_2$ whenever
    there exists $s \in \{1,\ldots,t-1\}$ such that $\bar k \in \bar \Gamma_s$
    and $\bar k \cdot \bar a_s > 0$, and let $H_1,H_2$ be fundamental domains of
    $T'$.  Then there exists a sequence $s_i \to \infty$ and a constant $c_0>0$
    such that
    $$|I(u_1,Z(s_i,\infty) \cap H_1) - I(u_2,Z(s_i,\infty) \cap H_2)|<
    c_0\,s_i^{j-2}\,.$$
  \end{lemma}

  \begin{proof} 
    For every $v \in \mcm(\bar a_1, \ldots, \bar a_{t-1})$ and every $r > 0$ and
    any two fundamental domains $H_1, H_2$ of $T'$ we have $I(v, Z(r, \infty)
    \cap H_1) = I(v, Z(r, \infty) \cap H_2)$. Thus, it suffices to give the
    proof for
    $$H_1 = H_2 = \{x \in \DR^n \mid 0 \le x\cdot k_0 < |k_0|\} =\mathrel{\mathop:}H_0\,.$$
    The idea is as follows: $\vol \big(Z(r,t)\big)$ grows like $ts^{j-1}$ and
    $\vol \big(Z(r,t) \cap H_0\big)$ grows like $s^{j-1}$. By ``short
    connections'' and minimality of $u_1$ and $u_2$ we obtain the desired
    estimate.

    For $n \in \DN$ we set $t_n \mathrel{\mathop:}= n|k_0|$. For every $r,n > 0$
    we let $g_{r,n}$ be the functions provided by Lemma \ref{l69} and set
    $\phi_{r,n} = g_{r,n} - u_1$.  Minimality of $u_1$ implies
    \begin{align*}
      I\big(u_1, Z(r,t_n)\big) &+ I\big(u_1, \spt \phi_{r,n} \setminus
      Z(r,t_n)\big) = I\big(u_1, Z(r,t_n) \cup
      \spt \phi_{r,n}\big) \\
      &\le I\big(u_1 + \phi_{r,n}, Z(r,t_n) \cup \spt \phi_{r,n}\big) \\
      &= I\big(u_2, Z(r,t_n)\big) + I\big(u_1 + \phi_{r,n}, \spt \phi_{r,n}
      \setminus Z(r,t_n)\big)\,.
    \end{align*}
    Hence
    \begin{align}\begin{split}
        \label{eq:k01}
        I\big(&u_1, Z(r,t_n)\big) -  I\big(u_2, Z(r,t_n)\big) \\
        &\le \big|I\big(u_1 + \phi_{r,n}, \spt \phi_{r,n} \setminus
        Z(r,t_n)\big) - I\big(u_1, \spt \phi_{r,n} \setminus
        Z(r,t_n)\big)\big|\,.
      \end{split}
    \end{align}

    By the assumption that $T_{\bar k} u_1 \ge u_2$ whenever there exists $s \in
    \{1,\ldots,t-1\}$ such that $\bar k \in \bar \Gamma_s$ and $\bar k \cdot
    \bar a_s > 0$, the set 
    $$W \mathrel{\mathop:}= \big\{ (x,x_{n+1}) \in \DR^{n+1} \mid u_1(x) <
    x_{n+1} < u_2(x) \big\}$$ projects injectively into $\DR^{n+1}/\bar
    \Gamma_t$. Furthermore, $W$ is $\DZ \bar k_0$-invariant and we obtain the
    following volume-growth estimate: There is a constant $\tilde c > 0$,
    independent of $n \in \DN$, such that
    \begin{align}\begin{split}\label{eq:k00}
        \vol\big(W \cap(Z(r,t_n) \times \DR\big) \le \tilde c\,n\,r^{j-1} + \tilde
        c\,r^j\,. \end{split}
    \end{align}
    Since the left hand side of this estimate equals the integral
    $$ \int_0^r\bigg(\int_{C_{t_n}(s)}(u_2-u_1)(x)\,d\sigma(x) +
    \int_{D_{t_n}(s)} (u_2-u_1)\,d\sigma(x)\bigg)\,ds\,,$$ Lemma \ref{volslice}
    yields a sequence $s_i \to \infty$ and a constant $c'>0$
    such that
    \begin{align*}
      \int_{C_{t_n}(s_i)} (u_2-u_1)\,d\sigma(x) &\le c'\,n
      s_i^{j-2}\qquad \mbox{and }\\
      \int_{D_{t_n}(s_i)} (u_2-u_1)\,d\sigma(x) &\le
      c'\,s_i^{j-1}\,
    \end{align*}
    for every $n \in \DN$.  By Lemma \ref{l69} (e) there is a constant $c'' >0$
    such that
    \begin{align*}\begin{split}\label{eq:k02}
        \big|I\big(u_1+\phi_{s_i,n},\spt \phi_{s_i,n}\setminus Z(s_i,t_n)\big)
        &-
        I\big(u_1,\spt \phi_{s_i,n}\setminus Z(s_i,t_n)\big)\big|\\
        &\le c'' \,n\,s_i^{j-2} + c''\, s_i^{j-1}\,.
      \end{split}
    \end{align*}
    Together with estimate \eqref{eq:k01} this implies
    $$  I\big(u_1, Z(s_i,t_n)\big) -  I\big(u_2, Z(s_i,t_n)\big)\le  c'' \,n\,s_i^{j-2} +
    c''\, s_i^{j-1}\,.$$ Using Remark \ref{e2} we infer
    \begin{equation}
      \label{eq:k06}
      \big| I\big(u_1, Z(s_i,t_n)\big) -  I\big(u_2, Z(s_i,t_n)\big) \big| \le  c'' \,n\,s_i^{j-2} +
      c''\, s_i^{j-1}\,.
    \end{equation}
    Consider a fixed $i \in \DN$. By the $\DZ \bar k_0$-invariance of $u_1$ and
    $u_2$ we obtain for $j = 1,2$
    \begin{equation*}
      I\big(u_j,Z(s_i,t_n)\big) = 2n \,I\big(u_j,Z(s_i,\infty) \cap H_0\big) + 2\,
      I\big(u_j, Z(s_i,t_n) \setminus \{ x \mid |x\cdot k_0| \le t_n\}\big)\,.
    \end{equation*}
    The modulus of the second term on the right hand side equals a constant
    $c^j$ depending on $s_i$ but not on $n$. Set $c''' =
    5\,\max\{|c^1|,|c^2|,c''s_i^{j-1}\}$, and infer from \eqref{eq:k06}
    \begin{align*}
      c'' \,n\,s_i^{j-2} + c'''&\ge 2n \, \big|I\big(u_1,Z(s_i,t_n) \cap
      H_0\big) - I\big(u_2,Z(s_i,t_n) \cap H_0\big)\big| \,.
    \end{align*}
    Considering $n \to \infty$, we infer $\big|I\big(u_1,Z(s_i,\infty)\cap
    H_0\big) - I\big(u_2,Z(s_i,\infty) \cap H_0\big)\big| \le c_0
    \,s_i^{j-2}\,.$
  \end{proof}
  
  \begin{proof}[Proof of Lemma \ref{glss}] 
    We define the sets
    \begin{align*}
      W' &\mathrel{\mathop:}= \left\{ (x,x_{n+1}) \in \DR^{n+1} \mid u(x) < x_{n+1} < w(x) \right\} \\
      W'' &\mathrel{\mathop:}= \left\{ (x,x_{n+1}) \in \DR^{n+1} \mid u(x) < x_{n+1} < T_{\bar
          k_0}u(x) \right\}\,,
    \end{align*}
    and consider the coverings
    \begin{equation}
      \label{eq:ac}
      \DR^{n+1} \stackrel{\pi}{\longrightarrow}\DR^{n+1}\big/\DZ \bar k_0
      \stackrel{p'}{\longrightarrow} \DR^{n+1}\big/\bar\Gamma_t\stackrel{\hat
        p}{\longrightarrow} T^{n+1}\,.  
    \end{equation}
    
    By Proposition \ref{4.2}(b) $\hat p$ maps $p' \big(\pi (W'')\big)$
    injectively into $T^{n+1}$.  The group of deck transformations of $p'$ is of
    rank $j-1$, thence  
    \begin{equation*}
      \vol_{n+1}\bigg(\pi\big(W'' \cap (Z(r,\infty) \times \DR)\big)\bigg) \le
      cr^{j-1}\,
    \end{equation*}
    for some constant $c > 0$. Since $\pi\big|_{W''}$ is injective and $W'
    \subset W''$, we have
    \begin{equation}\label{eq:a0}
      \vol_{n+1} \bigg(W'\cap \big(Z(r, \infty)\times\DR\big)\bigg)\le
      cr^{j-1}\,.
    \end{equation} 

    Now we fix the radius $s$ of $Z(s,t)$: Integration in cylindric coordinates
    and Lemma \ref{volslice} implies that there exists a sequence $s_i \to
    \infty$, and a constant $c>0$ such that
    \begin{equation}
      \label{eq:a14}
      \int_{C_\infty(s_i)} (w-u)(x)\,d\sigma(x) \le cs_i^{j-2}\,.      
    \end{equation}
    \begin{remark}
      Lemma \ref{volslice} allows us to choose the same sequence $s_i\to \infty$
      here and in Lemma \ref{vu}, and we do so.
    \end{remark}
    From now
    on let $i$ be fixed (but arbitrarily large) such that
    \begin{equation}
      \label{eq:sgross}
      s\mathrel{\mathop:}=s_i >\max\left\{\frac{8c\tilde A}\delta,\frac{8 c_0}\delta\right\}\,,  
    \end{equation}    
    where $c_0$ is the constant from Lemma \ref{vu} and $\delta$ from Lemma
    \ref{6.8}. Then
    \begin{align}
      \label{eq:a5}
      c_0s^{j-2} &< \frac \delta8 s^{j-1}\,. 
    \end{align}

    We fix the height $t$ of $Z(s,t)$: By \eqref{eq:a0},
    $\vol_{n+1}\bigg(W'\cap\big(Z(s,\infty) \times \DR\big)\bigg)<\infty$, Lemma
    \ref{volslice} yields a sequence $t_l \to \infty$ and a constant $\hat c >
    0$ with
    $$\int_{D_{t_l}(s)} (w-u)(x)\,d\sigma(x) \le \frac {\hat
      c}{t_l}\,.$$ This estimate together with \eqref{eq:a14} and Lemma
    \ref{l69} (e) yield functions $\phi_{i,l}$ with $(u +
    \phi_{i,l})\big|_{Z(s,t_l)} = w\big|_{Z(s,t_l)}$ and
    $\Delta\big(u,\phi_{i,l},Z(s,t_l)^C\big)< c\tilde A \,s^{j-1} + \frac{\hat c
      \tilde A}{t_l}$. We choose $l_0$ so large that for every $l \ge l_0$ we
    have $\frac{\hat c \tilde A}{t_l} < \frac \delta 8 s^{j-1}$.  Together with
    \eqref{eq:sgross} we infer
    \begin{equation}
      \label{eq:a2}
      \big|\Delta\big(u,\phi_{i,l},Z(s,t_l)^C\big)\big|< \frac {\delta}4s^{j-1}\,.  
    \end{equation}
    This estimates the ``energy costs of the short connections'' outside
    $Z(s,t_l)$.

    Now we will compare the energies of $u$ and $w$ inside $Z(s,t_l)$.  The set
    $$H\mathrel{\mathop:}= \big\{x \in \DR^n \mid u(x)<v(x)\le T_{\overline k_0}u\big\}$$ is a
    measurable fundamental domain of the action $T'$ of $\DZ k_0$ on $\DR^n$ and
    we consider two more measurable fundamental domains $H_l', H_l''$ that
    satisfy
    \begin{align*}
      &Z(s,t_l) \setminus T'_{-k_0} Z(s,t_l) \subset H_l'\quad \mbox{and}\\
      &Z(s,t_l) \setminus T'_{k_0} Z(s,t_l) \subset H_l''\,.
    \end{align*}
    By the convergence provided by Proposition \ref{4.2}(a) and by continuity of
    $F$ there exists an integer $l_1\ge l_0$ such that for every $l \ge l_1$:
    \begin{equation*}
      \big|I(u^-,Z(s,t_l)\cap H_l') - I(u,Z(s,t_l)\cap H_l')\big|
      < \frac \delta{8}  s^{j-1}\,.  
    \end{equation*} 
    Together with Lemma \ref{vu} and \eqref{eq:a5}, this implies that
    \begin{equation}
      \label{eq:a3}
      \big|I(v,Z(s,t_l)\cap H) - I(u,Z(s,t_l)\cap H_l')\big| < \frac
      {\delta}{4} s^{j-1}\,.
    \end{equation}
    
    By the assumed periodicity of $u^\pm$ and $v$, there exists a constant
    $\delta'>0$ such that $|u^\pm(x)-v(x)| >\delta'$ on $Z(s,\infty) \cap H_l'$
    for every $l \in \DN$.  Thus the above-mentioned convergence result implies
    that there exists an integer $l_2 \ge l_1$ such that for all $l \ge l_2$
    \begin{align}\begin{split}\label{eq:gss9}
        Z(s,t_l) \cap H_l' \cap \{u \ge v\} = \emptyset = Z(s,t_l) \cap H_l''
        \cap \{T_{\bar k_0} u < v\}\,.\end{split}
    \end{align}
    Set $K \mathrel{\mathop:}= Z(s,t_{l_2})$ and $\phi = \phi_s = \phi_{i,l_2}$ and observe
    \begin{align*}
      w\cdot\chi_K &= v\cdot\chi_{K\cap H} + u\cdot\chi_{K\cap \{u \ge v\}} + T_{\bar k_0}
      u\cdot\chi_{K \cap \{T_{\bar k_0}u<v\}}\\
      &= v\cdot\chi_{K\cap H} + u\cdot\chi_{K\cap H_l' \cap \{u \ge v\}} +
      u\cdot\chi_{(K\setminus H_l')\cap \{u \ge v\}} \\
      &\qquad T_{\bar k_0} u\cdot\chi_{K \cap H_l''\cap \{T_{\bar k_0}u<v\}} +
      T_{\bar k_0} u\cdot\chi_{(K \setminus H_l'')\cap \{T_{\bar k_0}u<v\}}\,. 
    \end{align*}
    Furthermore periodicity of $F$ yields
    \begin{equation}\label{eq:a6}
      I\big(T_{\bar k_0} u,  (K\setminus H_l'') \cap \{T_{\bar k_0}
      u < v\}\big) = I\big(u, (K\setminus H_l')\cap \{u < v\}\big) \,. 
    \end{equation}
    The above decomposition of $w\cdot\chi_K$ and \eqref{eq:a3}, \eqref{eq:gss9}
    and \eqref{eq:a6} gives
    \begin{align*}
      I\big(w,K\big) &< I\big(u,K\cap H_l'\big) + \frac
      {\delta}{4}s^{j-1} + 0+ I\big(u,(K\setminus
      H_l')\cap \{u \ge
      v\}\\
      &\qquad +0+ I\big(u, K\setminus H_l' \big)\cap \{u < v\}\big) \\
      &= I\big(u,K\big) + \frac \delta 4 s^{j-1}\,.
    \end{align*}
    Together with $(u+\phi)\cdot\chi_{K \cup \spt \phi} = w\cdot\chi_K +
    (u+\phi)\cdot\chi_{\spt \phi \setminus K}$ and \eqref{eq:a2}, this implies
    $I\big(u+\phi,K\cup \spt \phi\big) < I\big(u,K\cup \spt \phi\big) + \frac
    \delta 2 s^{j-1}$.
    \end{proof}
\bibliographystyle{siam}    
\bibliography{lit}            
\end{document}

%% file: makros.tex
\newtheorem{theorem} {\bf Theorem} [section]
\newtheorem{lemma}[theorem]{\bf Lemma}
\newtheorem{prop}[theorem]{\bf Proposition}

\theoremstyle{remark}

\newtheorem{remark}[theorem]{Remark}

\def\phi{\varphi}
\def\epsilon{\varepsilon}
\def\theta{\vartheta}

\newcommand{\mcl}{{\mathscr{L}}}

\newcommand{\mcm}{{\mathscr{M}}}

\newcommand{\graph}{{\mathrm{graph}}}

\newcommand{\Lip}{{\mathrm{Lip}}}

\newcommand{\spt}{\operatorname{\mathrm spt}}

\newcommand{\lspan}{\operatorname{\mathrm span}}

\newcommand{\vol}{\operatorname{\mathrm vol}}

\newcommand{\rank}{\operatorname{\mathrm rk}}

\newcommand{\pr}{\operatorname{\mathrm pr}}

\newcommand{\loc}{\mathrm{loc}}

\newcommand{\Bl}[1]{{\mathbb{#1}}}

\newcommand{\DR}{\Bl{R}}
\newcommand{\DZ}{\Bl{Z}}

\newcommand{\DN}{\Bl{N}}

\newboolean{comment}
\newcommand{\comline}[1]{\ifthenelse{\boolean{comment}}{{\bf
      \noindent\shortstack[r]{\rule{0cm}{0.5cm}\\
      #1\\\rule{16cm}{0.02\pagestyle{myheadings}
 \markboth{Links}{Rechts}
 5cm}}\\}}}

\newcommand{\comtxt}[1]{\ifthenelse{\boolean{comment}}{{\bf #1} \\}}

\newcommand{\commar}[1]{\ifthenelse{\boolean{comment}}{\marginpar{{#1}}}{}}

\setboolean{comment}{true}          

\newcommand{\btxvolumelong }[1]{\relax}

\newcommand{\btxofserieslong }[1]{\relax}

\newcommand{\btxandlong }[1]{and}
